\documentclass[10pt,a4paper]{article}
\usepackage{latexsym,amsthm,amssymb,amscd,amsmath}
\newcounter{alphthm}
\theoremstyle{plain}
\newtheorem{theorem}{Theorem}[section]
\newtheorem{lemma}[theorem]{Lemma}
\newtheorem{proposition}[theorem]{Proposition}
\newtheorem{cor}[theorem]{Corollary}
\theoremstyle{definition}
\newtheorem{definition}[theorem]{Definition}

\newcommand{\be}{\begin{equation}}
\newcommand{\ee}{\end{equation}}
\newcommand{\ben}{\begin{enumerate}}
\newcommand{\een}{\end{enumerate}}

\begin{document}
\title{On Soft Connectedness}
\author{E. Peyghan, B. Samadi, A. Tayebi}
\maketitle
\begin{abstract}
The soft topological spaces and some their related concepts have studied in \cite{SN}. In this paper, we introduce and study the notions of soft connected topological spaces after a review of preliminary definitions.  
\end{abstract}


\section{Introduction}
Some theories such as theory of vague sets, theory of rough sets and etc, can be considered as mathematical tools for dealing with uncertainties. But all of these theories have their own difficulties. Molodtsov \cite{M} introduced the concept of soft sets in order to solve complicated problems in some sciences such as, economics, engineering and etc. In fact the concept of soft sets is a new mathematical tool which is free from the difficulties mentioned above. 

In 2011, Shabir and Naz introduced and studied the concepts soft topological space and some related concepts such as soft interior, soft closed, soft subspace and soft separation axioms. 

This is a natural activity that a topologist wants to prove or disprove another important theorems of general topology of the soft set form. For instance in \cite{AA}, the authors introduced the soft product topology and defined the version of compactness in soft spaces named soft compactness. 

In this paper, we introduced some concepts such as soft connectedness, soft locally connectedness and we exhibit some results related to these concepts and soft product spaces.       
\section{Main results}
\begin{definition}
Let $(X, \tau, E)$ be a soft topological space over $X$. A soft
separation of $\widetilde{X}$ is a pair $(F, E)$, $(G, E)$ of
no-null soft open sets over $X$ such that
\[
\widetilde{X}=(F, E)\cup(G,
E),\ \ (F, E)\cap(G, E)=\Phi_E.
\]
\end{definition}
\begin{definition}
A soft topological space $(X, \tau, E)$ is said to be soft connected if
there does not exist a soft separation of $\widetilde{X}$.
\end{definition}
\begin{proposition}\label{1}
Let $(F, E)$ be a soft set in $SS(X)_E$. Then

(i)\ $(F, E)\cup (F, E)'=\widetilde{X}$,

(ii)\ $(F, E)\cap (F, E)'=\Phi_E$,

(iii)\ $(F, E)\cap \widetilde{X}=(F, E)$.
\end{proposition}
\begin{proof}
Let $(F, E)\cap (F, E)'=(H, E)$. Then $H(e)=F(e)\cup
F'(e)=F(e)\cup (X-F(e))=\emptyset$. Therefore $(H, E)=\Phi_E$.
\end{proof}
\begin{theorem}
A soft topological space $(X, \tau, E)$ is soft connected if and
only if the only soft sets in $SS(X)_E$ that are both soft open
and soft closed over $X$ are $\Phi_E$ and $\widetilde{X}$.
\end{theorem}
\begin{proof}
Let $(X, \tau, E)$ be soft connected. Suppose to the contrary that
$(F, E)$ is both soft open and soft closed in $X$
different from $\Phi_E$ and $\widetilde{X}$.Clearly, $(F, E)'$
is a soft open set in $X$ different from $\Phi_E$ and
$\widetilde{X}$. Now by Proposition \ref{1} we have $(F,
E)$, $(F, E)'$ is a soft separation of $\widetilde{X}$. This is
a contradiction. Thus the only soft closed and open sets in
$X$ are $\Phi_E$ and $\widetilde{X}$. Conversely, let
$(F, E)$, $(G, E)$ be a soft separation of $\widetilde{X}$. Let
$(F, E)=\widetilde{X}$. Then Proposition \ref{1} implies that $(G,
E)=\Phi_E$. This is a contradiction. Hence, $(F,
E)\neq\widetilde{X}$. Since $F(e)\cap G(e)=\emptyset$ and
$F(e)\cup G(e)=X$, for each $e\in E$, we have $G'(e)=X-G(e)=F(e)$.
Therefore $(F, E)=(G, E)'$. This shows that $(F, E)$ is both soft
open and soft closed in $X$ different from $\Phi_E$ and $\widetilde{X}$. This is a contradiction. Therefore, $(X,
\tau, E)$ is soft connected.
\end{proof}
\begin{definition}
Let $SS(U)_A$ and $SS(V)_B$ be families of soft sets. Let
$u:U\longrightarrow V$ and $p:A\longrightarrow B$ be mappings.
Then a mapping $f_{pu}:SS(U)_A\longrightarrow SS(V)_B$ is defined
as:

(i)\ Let $(F, A)$ be a soft set in $SS(U)_A$. The image of $(F,
A)$ under $f_{pu}$, written as $f_{pu}(F, A)=(f_{pu}(F), B)$ is a
soft set in $SS(V)_B$ such that,
\[
f_{pu}(F)(y)=\left\{
\begin{array}{ccc}
\bigcup_{x\in p^{-1}(y)\cap A} u(F(x))&p^{-1}(y)\cap
A\neq\emptyset\\
\emptyset &p^{-1}(y)\cap A=\emptyset
\end{array}
\right.
\]
for each $y\in B$.

(ii)\ Let $(G, B)$ be a soft set in $SS(V)_B$. Then the inverse image of $(G, B)$ under $f_{pu}$, written as $f^{-1}_{pu}(G, B)=(f^{-1}_{pu}(G), A)$, is a soft set in
$SS(U)_A$ such that, $f^{-1}_{pu}(G)(x)=u^{-1}(G(p(x)))$, for each $x\in A$.
\end{definition}
\begin{proposition}\label{babak}
Let $SS(U)_A$ and $SS(V)_B$ be families of soft sets. For a function $f_{pu}:SS(U)_A\longrightarrow SS(V)_B$ we have

(i)\ $f^{-1}_{pu}((F, B)\cup(G, B))=f^{-1}_{pu}(F, B)\cup f^{-1}_{pu}(G, B)$,

(ii)\ $f^{-1}_{pu}(\widetilde{V})=\widetilde{U}$,

(iii)\ $f_{pu}((F, A)\cap(G, A))\widetilde{\subseteq} f_{pu}(F, A)\cap f_{pu}(G, A)$,

(iv)\ $f^{-1}_{pu}((F, B)\cap(G, B))=f^{-1}_{pu}(F, B)\cap f^{-1}_{pu}(G, B)$,

(v)\ $f^{-1}_{pu}(\Phi_B)=\Phi_A$.
\end{proposition}\label{2}
\begin{proof}
(i)\ Let $(F, B)\cup (G, B)=(H, B)$. Then $f^{-1}_{pu}(H, B)=(f^{-1}_{pu}(H), A)$, where $f^{-1}_{pu}(H)(x)=u^{-1}(H(p(x)))$, for each $x\in A$. On the other hand, let
$f^{-1}_{pu}(F, B)\cup f^{-1}_{pu}(G, B)=(O, A)$, where
\[
O(x)=f^{-1}_{pu}(F)(x)\cup f^{-1}_{pu}(G)(x)=u^{-1}\Big(F(p(x))\cup G(p(x))\Big)=u^{-1}(H(p(x))),
\]
for each $x\in A$. Therefore $f^{-1}_{pu}(H, B)=(O, A)$.

(ii)\ $f^{-1}_{pu}(\widetilde{V})=f^{-1}_{pu}(V, B)=(f^{-1}_{pu}(V), A)$, where $f^{-1}_{pu}(V)(x)=u^{-1}(V(p(x)))=u^{-1}(V)=U=U(x)$.

(iii)\ Let $(F, A)\cap (G, A)=(H, A)$. Then $f_{pu}(H, A)=(f_{pu}(H), B)$, where
\[
f_{pu}(H)(y)=\left\{
\begin{array}{ccc}
\bigcup_{x\in p^{-1}(y)\cap A} u(H(x))&p^{-1}(y)\cap
A\neq\emptyset\\
\emptyset &p^{-1}(y)\cap A=\emptyset
\end{array}
\right.
\]
for each $y\in B$. On the other hand, let $f_{pu}(F, A)\cap f_{pu}(G, A)=(O, B)$, where $O(y)=f_{pu}(F)(y)\cap f_{pu}(G)(y)$, for each $y\in B$. We have
\[
O(y)=\left\{x\in A
\begin{array}{ccc}
(\bigcup_{x\in p^{-1}(y)\cap A} u(F(x)))\cap (\bigcup_{x\in p^{-1}(y)\cap A} u(G(x)))&p^{-1}(y)\cap
A\neq\emptyset\\
\emptyset &p^{-1}(y)\cap A=\emptyset
\end{array}
\right.
\]
for each $y\in B$. Since $H(x)=F(x)\cap G(x)$, for each $x\in A$, then it is easy to see that $f_{pu}(H)(y)\subseteq O(y)$ for each $y\in p^{-1}(y)\cap A$. This implies that $f_{pu}(H, A)\subseteq (O, B)$.

(iv)\ Let $(F, B)\cap (G, B)=(H, B)$. Then $f^{-1}_{pu}(H, B)=(f^{-1}_{pu}(H), A)$, where $f^{-1}_{pu}(H)(x)=u^{-1}(H(p(x)))$, for each $x\in A$. On the other hand, let $f^{-1}_{pu}(F, B)\cap f^{-1}_{pu}(G, B)=(O, A)$, where
\begin{eqnarray*}
O(x)\!\!\!\!&=&\!\!\!\!f^{-1}_{pu}(F)(x)\cap f^{-1}_{pu}(G)(x)=u^{-1}(F(p(x)))\cap u^{-1}(G(p(x)))\\
\!\!\!\!&=&\!\!\!\!u^{-1}(H(p(x))),
\end{eqnarray*}
for each $x\in A$.
Therefore, $f^{-1}_{pu}(H, B)=(O, A)$.
\end{proof}
\begin{definition}
Let $(U, \tau, A)$ and $(V, \tau', B)$ be soft topological spaces. Let $f_{pu}:SS(U)_A\longrightarrow SS(V)_B$ be a function. Then $f_{pu}$ is said to be soft $pu$- continuous if for each $(F, B)\in\tau'$ we have $f^{-1}_{pu}(F, B)\in\tau$.
\end{definition}
\begin{theorem}
Let $f_{pu}$ be a soft $pu$- continuous function carrying the soft connected space $(U, \tau, A)$ onto the soft space $(V, \tau', B)$. Then $(V, \tau', B)$ is soft connected.
\end{theorem}
\begin{proof}
Suppose to the contrary there exists a soft separation $(F, B)$, $(G, B)$ of $\widetilde{V}$. Then Proposition \ref{babak} implies that $\widetilde{U}=f^{-1}_{pu}((F, B)\cup(G, B))=f^{-1}_{pu}(F, B)\cup f^{-1}_{pu}(G, B)$ and $f^{-1}_{pu}(F, B)\cap f^{-1}_{pu}(G, B)=f^{-1}_{pu}(\Phi_B)=\Phi_A$. If $f^{-1}_{pu}(F, B)=\Phi_A$, since $f_{pu}$ is surjective, by Theorem ... and Proposition \ref{babak}, we have $
(F, B)=\Phi_B$. This is a contradiction. Therefore $f^{-1}_{pu}(F, B)$ and by a similar reason $f^{-1}_{pu}(G, B)$ are different from $\Phi_A$. This shows that $f^{-1}_{pu}(F, B)$, $f^{-1}_{pu}(G, B)$ is a soft separation of $\widetilde{U}$. This is a contradiction, and this completes the proof.
\end{proof}
\begin{definition}
Let $(F, E)$ be a soft set over $X$ and $Y$ be a nonempty subset of $X$. Then the sub soft set of $(F, E)$ over $Y$ denoted by $({}^YF, E)$ is defined as follows
\[
{}^YF(e)=Y\cap F(e),
\]
for each $e\in E$. In other word $({}^YF, E)=\widetilde{Y}\cap (F, E)$.
\end{definition}
\begin{definition}
Let $(X, \tau, E)$ be a soft topological space over $X$ and $Y$ be a nonempty subset of $X$. Then
\[
\tau_Y=\{({}^YF, E)|(F, E)\in\tau\},
\]
is said to be the soft relative topology on $Y$ and $(Y, \tau_Y, E)$ is called a soft subspace of $(X, \tau, E)$.
\end{definition}
\begin{proposition}\label{6}
Let $(F, E)$, $(G, E)$ and $(H, E)$ be soft sets in $SS(X)_E$. Then,

(i)\ $(F, E)\cap ((G, E)\cup(H, E))=((F, E)\cap(G, E))\cup((F, E)\cap(H, E))$,

(ii)\ $(F, E)\widetilde{\subseteq} (G, E)$ if and only if $(F, E)\cap (G, E)=(F, E)$.
\end{proposition}
\begin{proof}
(i)\ Let $(G, E)\cup (H, E)=(A, E)$ and $(F, E)\cap(A, E)=(B, E)$. Then, $B(e)=F(e)\cap A(e)=F(e)\cap (G(e)\cup H(e))=(F(e)\cap G(e))\cup(F(e)\cap H(e))$, for each $e\in E$. On the other hand, if $(F, E)\cap (G, E)=(C, E)$, $(F, E)\cap (H, E)=(D, E)$ and $(C, E)\cup (D, E)=(I, E)$. we have, $I(e)=C(e)\cup D(e)=(F(e)\cap
G(e))\cup (F(e)\cap H(e))$ for each $e\in E$. Therefore, $(B,E)=(I,E)$.
\end{proof}
\begin{proposition}\label{4}
If the soft sets $(F, E)$ and $(G, E)$ form a soft separation of $\widetilde{X}$, and $(Y, \tau_Y, E)$ is a soft connected subspace of $(X, \tau, E)$, then $\widetilde{Y}$ lies entirely within either $(F, E)$ or $(G, E)$.
\end{proposition}
\begin{proof}
Since $\widetilde{Y}\widetilde{\subseteq} (F,E)\cup (G,E)$ by Proposition \ref{6} we have, $\widetilde{Y}=(\widetilde{Y}\cap(F,E))\cup (\widetilde{Y}\cap(G,E))$ that  $\widetilde{Y}\cap (F,E)$ and $\widetilde{Y}\cap (G,E)$ are soft open sets over $Y$. Suppose to the contrary $\widetilde{Y}$ does not lie entirely within either $(F,E) $ or $(G,E)$. By the hypothesis, Proposition 3.4 of \cite{ZAMA} and Proposition \ref{6}, $(\widetilde{Y}\cap (F,E)$ and $(\widetilde{Y}\cap (G,E)$ are different from $\widetilde{Y}$ and $\Phi_Y$. But $Y(e)\cap F(e)\cap G(e)=\emptyset$, for each $e\in E$. Therefore, $(\widetilde{Y}\cap (F, E))\cap (\widetilde{Y}\cap(G, E))=\Phi_Y$. Since $(\widetilde{Y}\cap(F, E))$ and $(\widetilde{Y}\cap(G, E))$ are soft open sets over $\widetilde{Y}$, then we have a soft separation of $\widetilde{Y}$. This is a contradiction. This completes the proof.
\end{proof}
\begin{proposition}\label{7}
Let $(F, E)$, $(G, E)$ and $(H, E)$ be soft sets in $SS(X)_E$. Then

(i)\ $(F, E)\cap ((G, E)\cap (H, E))=((F, E)\cap(G, E))\cap(H, E)$,

(ii)\ $(F, E)\cup ((G, E)\cup (H, E))=((F, E)\cup(G, E))\cup(H, E)$.
\end{proposition}
\begin{proposition}\label{3}
Let $\{(F_\alpha, E)\}_{\alpha\in J}$ be a family of soft sets in $SS(X)_E$, then

(i)\ $(F, E)\cap(\cup_{\alpha\in J}(F_\alpha, E))=\cup_{\alpha\in J}((F, E)\cap(F_\alpha, E))$,

(ii)\ If $(F, E)=(G, E)\cup(H, E)$, then $(G, E), (H, E)\widetilde{\subset}(F, E)$.
\end{proposition}
\begin{lemma}\label{5}
Let $(Y, \tau', E)$ and $(Z, \tau'', E)$ be soft subspaces of $(X, \tau, E)$ and $(Y, E)\widetilde{\subseteq} (Z, E)$. Then $(Y, \tau, E)$ is a soft subspace of $(Z, \tau'', E)$.
\end{lemma}
\begin{proof}
By Proposition \ref{6}, we have $\widetilde{Y}=\widetilde{Y}\cap\widetilde{Z}$, moreover each soft open set of $(Y, \tau', E)$ is of the form $\widetilde{Y}\cap (F, E)$ where $(F, E)$ is a soft open set of $(X, \tau, E)$. Therefore, by Proposition \ref{7}, we have $\widetilde{Y}\cap(F, E)=(\widetilde{Y}\cap \widetilde{Z})\cap(F, E)=\widetilde{Y}\cap(\widetilde{Z}\cap(F, E))$. Conversely, it is clear that each soft open set in $Y$ as a soft subspace of $(Z, \tau'', E)$ is of the form $Y\cap(\widetilde{Z}\cap(F, E))=\widetilde{Y}\cap(F, E)$. This completes the proof.
\end{proof}
\begin{theorem}\label{babak3}
The union of a collection of soft connected subspace of $(X, \tau, E)$ that have non-null intersection is soft connected.
\end{theorem}
\begin{proof}
Let $\{(Y_\alpha, \tau_{Y_\alpha}, E)\}_{\alpha\in J}$ be an arbitrary collection of soft connected soft subspace of $(X, \tau, E)$. Suppose to the contrary that there exists a soft separation of $\widetilde{Y}=\cup_{\alpha\in J}\widetilde{Y}_\alpha$, $\widetilde{Y}\cap (F, E)$, $\widetilde{Y}\cap (G, E)$. By the Proposition \ref{3}, we have $\widetilde{Y}=(\cup_{\alpha\in J}(F_\alpha, E))\cup(\cup_{\alpha\in J}(G_\alpha, E))$ where $F_\alpha(e)=F(e)\cap Y_\alpha$ and $G_\alpha(e)=G(e)\cap Y_\alpha$, for each $\alpha\in J$ and $e\in E$. Since $\cap_{\alpha\in J}\widetilde{Y}_\alpha\neq\Phi_E$, it is easy to see that $\cap_{\alpha\in J}Y_\alpha\neq\emptyset$, and $x\in\cap_{\alpha\in J}Y_\alpha$. On the other hand Lemma \ref{5} implies that $(Y_\alpha, \tau_\alpha, E)$ is a soft subspace of $(Y, \tau_Y, E)$, for each $\alpha\in J$. By Proposition \ref{4}, we can assume that $\widetilde{Y}_\alpha$ lies entirely within $\widetilde{Y}\cap (F, E)$. Let $\alpha'\in J-\{\alpha\}$. If $\widetilde{Y}_{\alpha'}\widetilde{\subseteq}\widetilde{Y}\cap(G, E)$, it is easy to see that $x\in Y\cap G(e)$, also $x\in Y\cap F(e)$, for each $\alpha\in J$. This is a contradiction. Therefore $\widetilde{Y}_{\alpha}\widetilde{\subseteq}\widetilde{Y}\cap(F, E)$, for each $\alpha\in J$. Now we can see that $\widetilde{Y}\widetilde{\subseteq}\widetilde{Y}\cap(F, E)$. Proposition of \cite{ZAMA} implies that $\widetilde{Y}\cap(G, E)\widetilde{\subseteq}\widetilde{Y}\cap(F, E)$ and $\Phi_E=\widetilde{Y}\cap(G, E)$. This is a contradiction. This completes the proof.
\end{proof}
\begin{definition}
Let $(X, \tau, E)$ be a soft topological space and $B\subseteq\tau$. If every element of $\tau$ can be written as a union of elements of $B$, then $B$ is called a soft basis for the soft topology $\tau$. Each element of $B$ is called a soft basis element.
\end{definition}
\begin{definition}
Let $(F, E_1)$ and $(G, E_2)$ be soft sets in $SS(X)_{E_1}$ and $SS(Y)_{E_2}$, respectively. Then the cartesian product of $(F, E_1)$ and $(G, E_2)$ denoted by $(F\times G, E_1\times E_2)$ in $SS(X\times Y)_{E_1\times E_2}$ is defined as $(F\times G)(e_1, e_2)=F(e_1)\times G(e_2)$.
\end{definition}
\begin{proposition}\label{babak1}
Let $(F_1,E_1),(G_1,E_1)\in SS(X)_{E_1}$ and $(F_2,E_2),(G_2,E_2)\in SS(Y)_{E_2}$. Then\\
(i) $\Phi_{E_1}\times(F_2,E_2)=(F_1,E_1)\times\Phi_{E_2}=\Phi_{E_1\times E_2}$,\\
(ii)
$((F_1,E_1)\times(F_2,E_2))\cap((G_1,E_1)\times(G_2,E_2))=((F_1,E_1)\cap(G_1,E_1))\times((F_2,E_2)\cap(G_2,E_2))$.
\end{proposition}
\begin{proof}
(i) Let $\Phi_{E_1}=(\phi_1,E_1)$ and $\Phi_{E_2}=(\phi_2,E_2)$.
\begin{eqnarray}
(F_1\times\phi_2)(e_1,e_2)\!\!\!\!&=&\!\!\!\!F_1(e_1)\times\phi_2(e_2)=F_1(e_1)\times\emptyset=\emptyset\nonumber\\
\!\!\!\!&&\!\!\!\!=\emptyset\times F_2(e_2)=\phi_1(e_1)\times
F_2(e_2)=(\phi_1\times F_2)(e_1,e_2).\nonumber
\end{eqnarray}
This implies (i).\\
(ii) Let $(F_1\times F_2,E_1\times E_2)\cap(G_1\times
G_2,E_1\times E_2)=(H,E_1\times E_2)$,
$(F_1,E_1)\cap(G_1,E_1)=(I,E_1)$ and
$(F_2,E_2)\cap(G_2,E_2)=(J,E_2)$. Then
\begin{eqnarray}
H(e_1,e_2)\!\!\!\!&=&\!\!\!\!(F_1\times F_2)(e_1,e_2)\cap(G_1\times G_2)(e_1,e_2)=\nonumber\\
\!\!\!\!&&\!\!\!\!(F_1(e_1)\times F_2(e_2))\cap(G_1(e_1)\times G_2(e_2))=\nonumber\\
\!\!\!\!&&\!\!\!\!(F_1(e_1)\cap G_1(e_1))\times(F_2(e_2)\cap G_2(e_2))=\nonumber\\
\!\!\!\!&&\!\!\!\!I(e_1)\times J(e_2)=(I\times
J)(e_1,e_2).\nonumber
\end{eqnarray}
Therefore, $(H,E_1\times E_2)=(I,E_1)\times(J,E_2)$.
\end{proof}
\begin{proposition}
Let $(X, \tau_1, E_1)$ and $(Y, \tau_2, E_2)$ be soft spaces. Let $B=\{(F, E_1)\times (G, E_2)| (F, E_1)\in\tau_1,\ (G, E_2)\in\tau_2\}$ and $\tau$ be the collection of all arbitrary union of elements of $B$. Then $\tau$ is a soft topology over $X\times Y$.
\end{proposition}
\begin{proof}
We have $\Phi_{E_1}=(\phi_1, E_1)\in\tau_1$ and $\Phi_{E_2}=(\phi_2, E_2)\in\tau_2$. Then, by Proposition \ref{babak1}, $\Phi_{E_1}\times\Phi_{E_2}=\Phi_{E_1\times E_2}\in\tau$. Moreover $\widetilde{X}=(X,E_1)\in\tau_1$ and $\widetilde{Y}=(Y,E_2)\in\tau_2$. Then $\widetilde{X}\times\widetilde{Y}=(X\times Y,E_1\times E_2)$ such that, $(X\times Y)(e_1,e_2)=X(e_1)\times Y(e_2)=X\times Y$, for each $(e_1,e_2)\in E_1\times E_2$. Therefore, $\widetilde{X}\times\widetilde{Y}=\widetilde{X\times Y}\in\tau$. Let $(F,E_1\times E_2)$, $(G,E_1\times E_2)\in\tau$. There exist the elements $(F_\alpha,E_1)\times(G_\alpha,E_2)$, $(F_\beta,E_1)\times(G_\beta,E_2)$, $\alpha\in I$, $\beta\in J$, of $B$ such that $(F,E_1\times E_2)=\bigcup_{\alpha\in I}((F_\alpha\times G_\alpha,E_1\times E_2))$ and $(G,E_1\times E_2)=\bigcup_{\beta\in J}((F_\beta\times G_\beta,E_1\times E_2))$. Let $(H,E_1\times E_2)=(F,E_1\times E_2)\cap(G,E_1\times E_2)$. Then, we have
\begin{eqnarray*}
H(e_1,e_2)\!\!\!\!&=&\!\!\!\!F(e_1,e_2)\cap G(e_1,e_2)\nonumber\\
\!\!\!\!&=&\!\!\!\!(\bigcup_{\alpha\in I}(F_\alpha(e_1)\times G_\alpha(e_2)))\cap(\bigcup_{\beta\in J}(F_\beta(e_1)\times G_\beta(e_2)))\nonumber\\
\!\!\!\!&=&\!\!\!\!\bigcup_{\beta\in J}\Big[(\bigcup_{\alpha\in J}(F_\alpha(e_1)\times G_\alpha(e_2)))\cap(F_\beta(e_1)\times G_\beta(e_2))\Big]\nonumber\\
\!\!\!\!&=&\!\!\!\!\bigcup_{\beta\in J}\bigcup_{\alpha\in I}((F_\alpha(e_1)\times G_\alpha(e_2))\cap(F_\beta(e_1)\times G_\beta(e_2)))\nonumber\\
\!\!\!\!&=&\!\!\!\!\bigcup_{\beta\in J}\bigcup_{\alpha\in I}((F_\alpha(e_1)\cap F_\beta(e_1))\times(G_\alpha(e_2)\cap G_\beta(e_2)))\nonumber\\
\!\!\!\!&=&\!\!\!\!\bigcup_{\alpha\in I,\beta\in J}((F_\alpha\cap F_\beta)(e_1)\times(G_\alpha\cap G_\beta)(e_2))\nonumber\\
\!\!\!\!&=&\!\!\!\!\bigcup_{\alpha\in I,\beta\in J}(F_\alpha\cap F_\beta\times G_\alpha\cap G_\beta)(e_1,e_2).\nonumber
\end{eqnarray*}
This shows that
\begin{eqnarray*}
(H,E_1\times E_2)\!\!\!\!&=&\!\!\!\!\bigcup_{\alpha\in I,\beta\in J}((F_\alpha\cap F_\beta)\times(G_\alpha\cap G_\beta),E_1\times E_2)=\nonumber\\
\!\!\!\!&&\!\!\!\!
\bigcup_{\alpha\in I,\beta\in J}((F_\alpha\cap F_\beta,E_1)\times(G_\alpha\cap G_\beta,E_2)).
\end{eqnarray*}
This implies that $(H,E_1\times E_2)\in\tau$. Finally, It is obvious  that an arbitrary union of elements of $B$ is an elements in $\tau$. This completes the proof.
\end{proof}
\begin{definition}
Let $(X,\tau_1,E_1)$ and $(Y,\tau_2,E_2)$ be soft spaces. Then the soft space $(X\times Y,\tau,E_1\times E_2)$ as
defined in previous proposition is called soft product topological space over $X\times Y$.
\end{definition}
\begin{proposition}\label{babak2}
Let $(F,E_1)$ and $(G,E_2)$ be soft sets in $SS(X)_{E_1}$ and $SS(Y)_{E_2}$, respectively. Then,
\[
((F,E_1)\times(G,E_2))'=((F,E_1)'\times \widetilde{Y})\cup(\widetilde{X}\times(G,E_2)').
\]
\end{proposition}
\begin{proof}
Let $(F\times G,E_1\times E_2)'=((F\times G)',E_1\times E_2)$. Then,
\begin{equation}
(F\times G)'(e_1,e_2)=(X\times Y)-(F(e_1)\times G(e_2))=((X-F(e_1))\times Y)\cup(X\times(Y-G(e_2))).\nonumber
\end{equation}
On the other hand,
\begin{equation}
((F,E_1)'\times\widetilde{Y})\cup(\widetilde{X}\times(G,E_2)')=(F'\times Y,E_1\times E_2)\cup(X\times G',E_1\times E_2).\nonumber
\end{equation}
If we denote this soft set by $(H,E_1\times E_2)$ we have,
\begin{eqnarray}
H(e_1,e_2)\!\!\!\!&=&\!\!\!\!(F'\times Y)(e_1,e_2)\cup(X\times G')(e_1,e_2)=(F'(e_1)\times Y)\cup(X\times G'(e_2))\nonumber\\
\!\!\!\!&&\!\!\!\!=((X-F(e_1))\times Y)\cup(X\times(Y-G(e_2))).\nonumber
\end{eqnarray}
This completes the proof.
\end{proof}
\begin{cor}
Let $(F,E_1)$ and $(G,E_2)$ be soft closed set in soft topological spaces $(X,\tau_1,E_1)$ and $(Y,\tau_2,E_2)$, respectively. Then $(F,E_1)\times(G,E_2)$ is soft closed set in soft product space $(X\times Y,\tau,E_1\times E_2)$.
\end{cor}
\begin{proof}
It is obvious that $(F,E_1)'$, $\widetilde{X}$ are soft open sets in $(X,\tau_1,E_1)$ and $(G,E_2)'$, $\widetilde{Y}$ are soft open sets in $(Y,\tau_2,E_2)$. Now, Proposition \ref{babak2} implies that $((F,E_1)\times(G,E_2))'$ is soft open in $(X\times Y,\tau,E_1\times E_2)$. This completes the proof.
\end{proof}
\begin{definition}
Let $(X,\tau,E)$ be a soft topological space over $X$ and $x,y\in X$ such that $x\neq y$. If there exist soft open sets $(F,E)$ and $(G,E)$ such that $x\in(F,E), y\in(G,E)$ and $(F,E)\cap(G,E)=\Phi_E$, then $(X,\tau,E)$ is called a soft $T_2$- space a soft Hausdorff.
\end{definition}
\begin{proposition}\label{10}
Let $(F,E_1)$ and $(G,E_2)$ be soft sets in $SS(X)_{E_1}$ and
$SS(Y)_{E_2}$, respectively. If $x\in(F,E_1)$ and $y\in(G,E_2)$.
then $(x,y)\in(F,E_1)\times(G,E_2)$, and vice versa.
\end{proposition}
\begin{proof}
By the hypothesis $x\in\bigcap_{e_1\in E_1}F(e_1)$ and $y\in\bigcap_{e_2\in E_2}G(e_2)$. Therefore,
\begin{eqnarray}
(x,y)\in(\bigcap_{e_1\in E_1}F(e_1))\times(\bigcap_{e_2\in E_2}G(e_2))\!\!\!\!&=&\!\!\!\!\bigcap_{(e_1,e_2)\in E_1\times E_2}(F(e_1)\times G(e_2))\nonumber\\
\!\!\!\!&&\!\!\!\!=\bigcap_{(e_1,e_2)\in E_1\times E_2}(F\times
G)(e_1,e_2).\nonumber
\end{eqnarray}
This shows that $(x,y)\in(F,E_1)\times(G,E_2)$. Conversely is
similar.
\end{proof}
\begin{proposition}
The product of two soft Hausdorff spaces is soft Hausdorff.
\end{proposition}
\begin{proof}
Let $(X,\tau_1,E_1)$ and $(Y,\tau_2,E_2)$ be soft Hausdorff spaces. we consider distinct points $(x_1,y_1)$ and $(x_2,y_2)$ of $X\times Y$. Without loss of generality let $x_1\neq x_2$. Then there exist soft open sets $(F,E_1)$ and $(G,E_1)$ in $(X,\tau,E_1)$ such that $x_1\in(F,E_1), x_2\in(G,E_1)$ and $(F,E_1)\cap(G,E_1)=\Phi_{E_1}$.
By Proposition \ref{10} we have,
$(x_1,y_1)\in(F,E_1)\times\widetilde{Y}$ and
$(x_2,y_2)\in(G,E_1)\times\widetilde{Y}$. These soft sets are soft
open in $(X\times Y,\tau,E_1\times E_2)$. Finally Proposition
\ref{babak1} shows that
\begin{equation}
((F,E_1)\times\widetilde{Y})\cap((G,E_1)\times\widetilde{Y})=\Phi_{E_1\times E_2},\nonumber
\end{equation}
 and this completes the proof.
\end{proof}
\begin{proposition}\label{8}
Let $\{(F_\alpha, B)\}_{\alpha\in J}$ be an arbitrary family of
soft sets in $SS(V)_B$. Then, $f^{-1}_{pu}(\cup_{\alpha\in
J}(F_\alpha, B))=\cup_{\alpha\in J}f^{-1}_{pu}(F_\alpha, B)$
\end{proposition}
\begin{proof}
Let $\cup_{\alpha\in J}(F_\alpha, B)=(F, B)$, where
$F(b)=\cup_{\alpha\in J}F_\alpha(b)$, for each $b\in B$. Then
$f^{-1}_{pu}(F, B)=(f^{-1}_{pu}(F), A)$, where
\[
f^{-1}_{pu}(F)(a)=u^{-1}(F(p(a)))=u^{-1}(\cup_{\alpha\in
J}F_\alpha(p(a)))=\cup_{\alpha\in J}u^{-1}(F_\alpha(p(a))),
\]
for each $a\in A$. On the other hand if 
\[
\cup_{\alpha\in J}f^{-1}_{pu}(F_\alpha, B)=\cup_{\alpha\in
J}(f^{-1}_{pu}(F_\alpha), A)=(G, A),
\]
then,
\[
G(a)=\cup_{\alpha\in J}f^{-1}_{pu}(F_\alpha)(a)=\cup_{\alpha\in
J}u^{-1}(F_\alpha(p(a))),
\]
for each $a\in A$. This completes the proof.
\end{proof}
\begin{lemma}
Let the soft topological space $(V, \tau', B)$ is given by soft
basis $B$. Then to prove soft $pu$- continuity of $f_{pu}$ it
suffices to show that the inverse image of every soft basis
element is soft open.
\end{lemma}
\begin{proof}
We consider $f_{pu}: SS(U)_A\longrightarrow SS(V)_B$. Let $(F, B)$
be a soft open set in soft space $(V, \tau', B)$. We can write
$(F, B)=\cup_{\alpha\in J}(F_\alpha, B)$, where $B=\{(F_\beta,
B)\}_{\beta\in I}$ is a soft basis of $(V, \tau, B)$ and
$J\subseteq I$. By Proposition \ref{8} we have
\[
f^{-1}_{pu}(F, B)=\cup_{\alpha\in J}f^{-1}_{pu}(F_\alpha, B),
\]
that is a soft open set in $(U, \tau, A)$.
\end{proof}
\begin{theorem}
The soft cartesian product of soft connected space $(U, \tau_1,
A)$ and $(V, \tau_2, B)$ is soft connected.
\end{theorem}
\begin{proof}
We choose a base point $(u', v')$ in the product $U\times V$. We
can see that the soft space $(U\times\{v'\}, \tau_{v'}, A\times
B)$ is soft connected. Otherwise, there exists a soft separation
$(F, A)\times (\widetilde{\{v'\}}, A)$ and $(G, A)\times (\widetilde{\{v'\}}, A)$ of
$\widetilde{U\times\{v'\}}$, that implies $(F, A)$,
$(G, A)$ is a soft separation of $\widetilde{U}$. This is a 
contradiction. Therefore $(U\times\{v'\}, \tau_{v'}, A\times B)$
is soft connected. By a similar reason $(\{u\}\times V, \tau_u,
A\times B)$ is soft connected, for each $u\in U$. As a result each
soft space $T_u=\big((U\times\{v'\})\cup(\{u\}\times V), \tau_{v'u},
A\times B\big)$ is soft connected being the union of two soft
connected soft subspace that have non-null intersection. It is
easy to see that $(U\times V, \tau, A\times B)$ 
is soft connected, where $U\times V=\bigcup_{u\in U}\big((U\times\{v'\})\cup(\{u\}\times U)\big)$, because it is the union of a collection of soft
connected subspace that have non-null intersection
$(\{u'\}\times \{v'\})$.
\end{proof}
\begin{definition}
Let $(X, \tau_1, E)$ and $(X, \tau_2, E)$ be soft topological
spaces. Then

(i)\ If $\tau_1\subseteq\tau_2$, then $\tau_2$ is soft finer than
$\tau_1$.

(ii)\ If $\tau_1\subset\tau_2$, then $\tau_1$ is soft strictly
finer than $\tau_2$.

(iii)\ If $\tau_1\subseteq\tau_2$ or $\tau_2\subseteq\tau_1$, then
$\tau_1$ is soft comparable than $\tau_2$.
\end{definition}
\begin{proposition}
Let $(X, \tau_2, E)$ be a soft connected space and
$\tau_1\subseteq \tau_2$. Then $(X, \tau_1, E)$ is soft connected.
\end{proposition}
\begin{proof}
Suppose to the contrary that $(F, E)$, $(G, E)$ is a soft
separation of $\widetilde{X}$ with soft topology $\tau_1$. Since
$\tau_1\subseteq\tau_2$, then $(F, E)$ and $(G, E)$ is a soft
separation of $\widetilde{X}$ with soft topology $\tau_2$. This is
a contradiction. Therefore $(X, \tau_1, E)$ is soft connected.
\end{proof}
We define a relation on $X$ ( or $\widetilde{X}$) by setting
$x\sim y$ if there is a soft connected subspace of $X$ such as
$(Y, \tau_Y, E)$ such that $x, y\in(Y, E)$, In other words $x, y\in Y$. Symmetry and
reflexivity of the relation are obvious. About transitivity, let
$x\sim y$ and $y\sim z$. Then there are soft connected subspace of
$X$ such as $(Y, \tau_Y, E)$ and $(Z, \tau_Z, E)$ such that $x,
y\in \widetilde{Y}$ and $y, z\in\widetilde{Z}$. This shows that
$\widetilde{Y}\cap\widetilde{Z}\neq\Phi_E$. Therefore, there
exists a soft connected subspace $(Y\cup Z, \tau_{Y\cup Z,}, E)$
such that $x, z\in\widetilde{Y\cup Z}$. Hence, $\sim$ is an equivalence relation.
\begin{definition}
The equivalence classes of the relation $\sim$ are called the soft components ( or the soft connected components ) of $\widetilde{X}$.
\end{definition}
If we denote the class of $x$ by $C$, it is easy to see that $C$ is the union of soft connected subspace of $(X, \tau, E)$ that they have $x$ in their intersection.
Theorem \ref{babak3} implies that the soft components are soft connected. Since $C=\bigcup_{x\in A_x}Ax$, where $(A_x,\tau_x,E)$ is soft connected subspace, for each $x\in X$. We can see that each soft connected subspace $(Y, E)$ lies entirely within just one of $C$.
\begin{theorem}\label{babak4}
The soft components of $(X,\tau,E)$ are soft connected disjoint subspace of $(A_x,\tau,E)$ whose union is $\widetilde{X}$, such that each soft connected subspace
lies entirely within just one of them.
\end{theorem}
\begin{definition}
A soft topological space $(X, \tau, E)$ is said to be soft locally connected at $x\in X$ if for every soft open set $(F, E)$ containing $x$, there is a soft connected soft open $(G, E)$ containing $x$ contained in $(F, E)$. If $(X, \tau, E)$ is soft locally connected at each of its points, it is said simply to be soft locally connected.
\end{definition}
\begin{theorem}\label{9}
A soft topological space $(X, \tau, E)$ is soft locally connected if and only if for every soft open set $(Y, E)$ of $X$, each soft component of $(Y, \tau_{Y}, E)$ is soft open in $X$.
\end{theorem}
\begin{proof}
Let $C$ be a soft component of $(Y, \tau_{Y}, E)$. If $x\in C$, there exists a soft connected soft open $(F_x, E)$ such that $x\in(F_x, E)\widetilde{\subseteq} C$. It is easy to see that $C=\cup_{x\in C}(F_x, E)$. Therefore $C$ is soft open in $X$. Conversely, let $(Y, E)$ be a soft open set containing $x$. Let $C$ be the soft component of $(Y, \tau_{Y}, E)$ containing $x$. By the hypothesis $C$ is soft open and by Theorem \ref{babak4}, $(C, \tau_{C}, E)$ is soft connected. Therefore, $(X, \tau, E)$ is soft locally connected.
\end{proof}
\begin{definition}
A soft topological space $(X, \tau, E)$ is said to be soft weakly locally connected at $x\in X$ if for every soft open set $(Y, E)$ containing $x$, there is a soft connected subspace of $(X, \tau, E)$ contained in $(Y, E)$ that contains a soft open set $(F, E)$ containing $x$ and similar to the soft locally connectedness, if $(X, \tau, E)$ is soft weakly locally connected at each of its points, it is said simply to be soft weakly locally connected.
\end{definition}
\begin{theorem}
If $(X, \tau, E)$ is soft weakly locally connected. Then it is soft locally connected.
\end{theorem}
\begin{proof}
By Theorem \ref{9} it is sufficient to show that for every soft open set $(Y, E)$, each soft component of $(Y, \tau_{Y}, E)$ is soft open. Let $C$ be a soft component of $(Y, \tau_{Y}, E)$ containing $x$.Then, there are a soft connected subspace $(Z, \tau_Z, E)$ and a soft open set $(F_x, E)$ containing $x$ such that $x\in(F_x, E)\widetilde{\subseteq} (Z, E)\widetilde{\subseteq}C\widetilde{\subseteq}(Y, E)$. This shows that $C=\cup_{x\in C}(F_x, E)$. This completes the proof.
\end{proof}

\noindent
Esmaeil Peyghan and Babak Samadi\\
Department of Mathematics, Faculty  of Science\\
Arak University\\
Arak 38156-8-8349,  Iran\\
Email: epeyghan@gmail.com
\bigskip

\noindent
Akbar Tayebi\\
Department of Mathematics, Faculty  of Science\\
University of Qom \\
Qom. Iran\\
Email:\ akbar.tayebi@gmail.com


\begin{thebibliography}{MaHo}
\bibitem{AA} A. Aygunoglu and H. Aygun, {\it Some notes on soft topological spaces}, Neural Comput and Applic, (2011), 1-7.
\bibitem{CKE} N. Cagman, S. Karatas and S. Enginoglu, {\it Soft topology}, Computers and Mathematics with Applications, {\bf 62} (2011), 351-358.
\bibitem {HA} S. Hussain and B. Ahmad, {\it Some properties of soft topological spaces}, Computers and Mathematics with Applications, {\bf 62} (2011), 4058-4067.
\bibitem{MBR} P. K. Maji, R. Biswas and A. R. Roy, {\it Soft set theory}, Computers and Mathematics with Applications, {\bf 45} (2003), 555-562.
\bibitem{M} D. A. Molodtsov, {\it Soft set theory- first results}, Computers and Mathematics with Applications, {\bf 37} (1999), 19-31.
\bibitem{WKM} W. K. Min,{\it A note on soft topological spaces}, Computers and Mathematics with Applications, {\bf 62} (2011), 3524-3528.
\bibitem {SN} M. Shabir and M. Naz, {\it On soft topological spaces}, Computers and Mathematics with Applications, {\bf 61} (2011), 1786–1799.
\bibitem{ZAMA} I. Zorlutuna, M. Akdag, W. K. Min, S. Atmaca, {\it Remarks on soft topological spaces}, Annals of Fuzzy Mathematics and Informatics, (2011),
\end{thebibliography}
\end{document}